\documentclass[titlepage,12pt,notitlepage]{amsart}
\usepackage{amsaddr}                                                                                                                                                                                                                                                                                                                                                                                                                                                     \usepackage{amsfonts,amsmath,amssymb}
\usepackage{latexsym,amscd,graphicx}
\usepackage[cmtip,all]{xy}
	\newtheorem{thm}{Theorem}[section]
\theoremstyle{definition}

\newtheorem{prop}[thm]{Proposition}

\newcommand{\s}{\mathfrak{sl} _2(\mathbb{C})}

\newcommand{\beq}{\begin{equation}}
\newcommand{\eeq}{\end{equation}}
\newcommand{\beas}{\begin{eqnarray*}}
\newcommand{\eeas}{\end{eqnarray*}}
\newcommand{\FF}{\mathbb{F}}

\DeclareMathOperator{\Supp}{\mathrm{Supp}}

\newcommand{\ad}{\mathrm{ad}}

\newtheorem{theorem}{Theorem}[section]

\newtheorem{lemma}[theorem]{Lemma}
\newtheorem{corollary}[theorem]{Corollary}

\newtheorem{definition}[theorem]{Definition}
\usepackage{lineno}
\usepackage{array}
\newcolumntype{L}{>{$}l<{$}}

\makeatother

\title{The title}%

\author{Abdallah Shihadeh}

\address[A1,A2]{Department of Mathematics, Faculty of Scince, The Hashemite University, Zarqa 13133, PO box 330127, Jordan}
\email[Abdallaha]{abdallaha\_ka@hu.edu.jo}

\usepackage{lipsum}

\begin{document}

\title{\sc{Graded prime ideals over graded Lie algebras }}

\thanks{{\em Keywords:} graded prime Lie ideals, semiprime Lie ideals, total prime Lie ideals, graded Lie algebras}
\thanks{{\em 2010 Mathematics Subject Classification:} Primary 17B70, Secondary 13A15}
\maketitle
\begin{abstract} 
	In this work, we extend the definition of the graded prime ideals from those in commutative graded rings to the ideals over graded Lie algebras. We prove some facts about graded prime Lie ideals in arbitrary Lie algebras that are similar to those about graded prime ideals over a commutative or non-commutative ring.In addition, the ideas of graded semiprime Lie ideals and graded total prime Lie ideals will be introduced.
\end{abstract}

\section{Introduction}\label{sI}

Consider $G$ an abelian group and a Lie algebra $ L $ over a field $\FF$. A  \emph{$G$-grading}  $L$ is the form of a direct sum decomposition of subspaces of $ L $
\begin{equation}\label{eI1}
	L=\bigoplus_{g\in G}L_g
\end{equation}  so that $ [L_g,L_h]\subseteq L_{g+h}$, for all $ g,h\in G $. In this case $L$ is said to be $G$-graded or $ L $ has $ G $-grade.

Keep in mind that zeros subspaces are allowed for the $L_g$. The grading's \textit{support} which is represented with  $\Supp L$ is a 
subset $S\subseteq G$ comprising of $g\in G$ where $L_g\ne\{0\}$ . 
The  \textit{homogeneous components} of the grading are the subspaces $L_g$ , and also the  elements in $L_g$ that are non-zero are referred as \textit{homogeneous of degree} $g$. Moreover,  the set of such homogeneous elements in $ L $ is presented as $ h(L)=\underset{ g\in  G}{\bigcup}L_g $.

	Suppose $G,\ G^{\prime}$ represent two abelian groups. Assuming that $L$ and $ M $ are  Lie algebras with $G$-grad and  $G^{\prime}$-grad respectively. 
	\begin{center}
		$ L=\underset{g\in G}{\bigoplus}L_g$\quad \text{and}\quad
		$M=\underset{h\in G^{\prime}}{\bigoplus} M_{h}$
	\end{center}
Suppose $\phi:L\longrightarrow M$ be a linear map.  $\phi$ is said to be \emph{graded} if for all $g\in G$, we have  $h\in G^{\prime}$ so that  $\phi(L_g)\subseteq M_{h}$. 

 An ideal $ I $ of $ L $ is considered to be a $ G $-graded ideal if
	\begin{equation}
		I=\underset{g\in G}{\bigoplus}(L_{g}\cap I).
	\end{equation} 

The aforementioned definition implies that for every $ z\in I $ with $ z=z_{g_{1}}+z_{g_{1}}+\cdots+z_{g_{n}} $, for some homogeneous elements $ z_{g_{i}}\in L_{g_{i}}  $, it follows that $ z_{g_{i}}\in I  $, i.e.  $ I $ generated by homogeneous elements.

The graded Lie algebras and graded Lie ideals and modules are wildly studied (see e.g. \cite{ BS2, BS3,BKS}).

Suppose  $ R $ be any commutative ring with unity, and $ P $ is an ideal of $ R $. The ideal $ P $ is considered to be prime ideal if for every $ a,b \in R$ with $ a.b\in P $, it follows that $ a\in P \text{ or } b\in P$.

 The prime ideal over Lie algebras is introduced in  \cite{Naoki}. An ideal $ P $  is considered to be a prime ideal of $ L  $ if for any two ideals $ I,\ J $ of $ L $ with $ [I,\ J]\subseteq P $, implies that $ I\subseteq P $ or $ J\subseteq P $. In this paper we will introduce a definition of graded prime ideals over $ L $.

 The following theorem due to N.Kawamoto \cite[Theorem 1]{Naoki} 
\begin{theorem}\cite[Theorem 1]{Naoki} 
The following are equivalent to an ideal $ P $ of $ L $:
	\begin{enumerate}
		\item P is prime 
		\item For $ a\in L $ and $ I $ an ideal of $ L $. If $\left[ a,I \right] \subseteq P$ then $ a\in P $ or $ I\subseteq L $ .
		\item For $ a,b \in L $. If $  \left[ a,\left\langle b \right\rangle  \right]\subseteq L  $, then $ a\in P $ or $ b\in P $.
	\end{enumerate}
\end{theorem}

The following ideal and its description is very important in our works.

\begin{definition}
	The ideal generated by $ x\in L $ is the smallest ideal of $ L $ containing the element $ x $, which is denoted by $ \left\langle x\right\rangle  $.
\end{definition}

In \cite{Naoki}, the author describe the ideal $ \left\langle x \right\rangle  $ as the following. Let
\begin{align}\label{generated }
\nonumber	V_0&=\mathbb{F}a \\ 
	V_i&=\left[ V_{i-1},L\right], \  \text{for all   }  i\geq 1 . 
\end{align} Then  $ \left\langle x \right\rangle=\overset{\infty}{\underset{i=0}{\sum}} V_i $.

From the previous description, it is obvious if $ x $ is homogeneous, then the ideal $  \left\langle x \right\rangle $ is graded.

Now let $ I, J $ be an ideal in $ L $. define the ideal 

\begin{align}
	\left(I:J \right)=&\left\lbrace y\in L \mid \left[ y,J\right] \subseteq I \right\rbrace \\ =& \left\lbrace y\in L \mid \ad_y(J) \subseteq I \right\rbrace\nonumber
\end{align}

For $ x\in L $ the ideal 
\begin{align}
	\left(I:x \right)=&\left\lbrace y\in L \mid \left[ y,x\right] \in I \right\rbrace
\end{align}

In this paper, we will introduce a definition of  the graded prime ideals, graded total prime ideals, and graded semiprime ideals and extend some facts about it from the case of commutative and non-commutative rings to the case of Lie algebras. In more than one place, we will follow some results in \cite{g.p submod,Naoki,semi Dawwas,g.p submod non comm ,g radicals } and extend it to our case.

\section {Graded prime Lie ideals }

From now on, let $ L $ be a G-graded Lie algebra over an arbitrary field $ \mathbb{F} $. In this section, we will define the graded prime and semiprime Lie ideals. We will provide some facts about these concepts.
\begin{definition}
Consider $ P $ be an ideal of $ L $. We said that $ P $ is graded prime ideal of $ L $, if whenever $ \left[ I,J\right] \subseteq P $ implies $ I\subseteq P $ or $ J\subseteq P $, for all graded ideals $ I,\ J  $ of $ L $.
\end{definition}

Keep in mind that if $ P $ is graded prime and $ x,y\in h(L) $ with $ [x,y]\in P $, implies that $ [\left\langle x\right\rangle,\left\langle y\right\rangle ] \subseteq P  $. Thus, $\left\langle  x\right\rangle \subseteq P $ or  $\left\langle  y\right\rangle \subseteq P $. That is $ x\in P $ or $ y\in P $. Note that the converse is not true.

\begin{theorem}
Consider the graded ideal $ P $ of $ L $. Then the following statements are equivalent:
	\begin{enumerate}
		\item $ P $ is graded prime ideal.
	\item If $  x\in h(L)$ and $ J $ be a graded ideal, with $ [ x , J  ]\subseteq P $, implies $ x\in P $ or $ J\subseteq P $.
		\item If $  x,y\in h(L)$ with $ [ x , \left\langle y \right\rangle ]\subseteq P $, implies $ x\in P $ or $ y\in P $.
	
	\end{enumerate}
\end{theorem}

\begin{proof}

	$ (1)\Rightarrow (2) $

Assume that $ \left[ x,J\right] \subseteq P$ where $ x\in h(L) $ and a graded ideal $ J $ of $ L $. Indeed   $ \left[ \left\langle x\right\rangle ,J\right] \subseteq P$, to see that assume that $ \left\langle x \right\rangle=\underset{i=0}{\overset{\infty}{\sum}}V_i  $, where $ V_i $ as in equation  (\ref{generated }). We are now using induction to show that $ \left[V_i,J \right]\subseteq P $. It is clear for $ i=0 $ that $ \left[V_0,J \right]  =\left[\mathbb{F}x,J \right]\subseteq P$.  Assume now that is true for $ i-1 $. Then 
	\begin{align*}	
		[V_i,J ]=[[V_{i-1},L ] J ]  \subseteq [[L, J], V_{i-1} ]+[[J,V_{i-1} ], L ]\subseteq[J,V_{i-1} ]+[ P,L]  \subseteq P .\end{align*}

	Hence, we have that 
	$ [ \left\langle x \right\rangle  , J] \subseteq P  $. Given that $ P $ indicates a graded prime we have that  $ \left\langle x\right\rangle \subseteq P  $ or $ J \subseteq P  $ and so $ x\in P $ or $ J \subseteq P  $.

	$ (2)\Rightarrow(3) $ Clear

	$ (3)\Rightarrow(2) $
	
	Let $ x\in h(L) $ and $ J $ represent a graded ideal with $ [x,J]\subseteq P $. If $ x\in P $, then we have done. Now if $ x\notin P  $, then for any homogeneous element $ y\in h(J)\subseteq h(L) $, we have that $ [x,y]\in P $ which implies that $ [x,\left\langle y\right\rangle ] \subseteq P $. Using part (3) and since $ x\notin P $, we have that $ y\in P $. Which implies that $ J\subseteq P $.

	$ (2)\Rightarrow(1) $

	Assume $ I,J $ be two graded ideals in $ L $ with  $ \left[ I,J \right] \subseteq P$. Assume that $ I \nsubseteq P $, hence there is a homogeneous element $ x\in h(I)-P\subseteq h(L)-P $. Since $ [I,J]\subseteq P $, we have that $ [x,J]\subseteq P $. By using part(2) we obtain that $ x\in P $ or $ J\subseteq P $. But $ x\notin P $ then $ J\subseteq P $.

\end{proof}

\begin{lemma}
	Assume $ P  $ denotes a graded ideal. Then for all $ x\in L(h)-P $, the ideal $ (P:x) $ is graded.
\end{lemma}

\begin{proof}
	Suppose $ x\in L(h)-P $. Since $ x $ is homogeneous in $ L $, then $ x\in L_h $ for some $ h\in G $.
	To see that   $ (P:x) $ is graded, we have to prove that $ (P:x)=\bigoplus_{g\in G}((P:x)\cap L_g) $.
	
	First we will show that $ (P:x)\cap L_g=(P_{gh}:x) $ for any $ g\in G $.
	Now let $ g\in G $ and 	let $ y_g\in (P:x)\cap L_g $. That is $ y_g $ is homogeneous of degree $ g $ and $ [y_g,x]\in P $. Since $ x $ is homogeneous of degree $ h $ in $ L $ then $ [y_g,x]\in L_{gh} $. But $ P $ is graded, that is $ [y_g,x]\in P_{gh} $ and so $ y\in (P_{gh}:x) $, implies that $ (P:x)\cap L_g\subseteq(P_{gh}:x) $.
	
	Now Let $y \in  (P_{gh}:x) \subseteq (P:x)$, then $ [y,x]\in P_{gh} \subseteq L_{gh} $. Since $ x $ is homogeneous of degree $ h $ in $ L $, then $ y $ must be a homogeneous of degree $ g $ in $ L $. Which implies that $ y\in (P:x)\cap L_g   $.

It is sufficient to demonstrate that $ (P:x)=\bigoplus_{g\in G} (P_{gh}:x) $. It is evident that $ (P:x)\supseteq\bigoplus_{g\in G} (P_{gh}:x) $.
	
	Let $ y\in (P:x) $. Since $ L $ is graded then $ y= y_{g_1}+y_{g_2}+\cdots+y_{g_k} $ for some $ g_1,g_2,\ldots,g_k\in G $. Now $ [y,x]=[y_{g_1},x]+[y_{g_2},x]+\cdots+[y_{g_k}+x]\in P $, and $ P $ is graded then for all  $ i=1,\ldots k $ we have that  $ [y_{g_i},x]\in P_{\bar{g}} $ for some $ \bar{g}\in G $. That is $ [y_{g_i},x]\in P_g  $ where $ g=\bar{g}hh^{-1} $. Which implies that $ [y,x]\in \bigoplus_{g\in G} (P_{gh}:x ) $, that is  $ (P:x)\subseteq\bigoplus_{g\in G} (P_{gh}:x) $. Which complete the proof.

\end{proof}

\begin{theorem}\label{P g.p iff (P:x)=p}
	Suppose $ P $ considered to be a graded ideal of $ L $. Then $ P $ is graded prime if and only if  $ (P:x)=P $ for all $ x\in h(L)-P $.
	 
\end{theorem}

\begin{proof}
	$ 	\left(\Rightarrow \right)  $

	it is evident that $ P\subseteq (P:x) $. Now let $ y\in (P:x) $. Case1: Let $ y\in h(L) $. Now since  $ y\in (P:x) $ we have $ [x,y]\in P $ which indicates that $ [\left\langle x\right\rangle ,\left\langle y\right\rangle]  \subseteq P$. Given $ P $ is graded prime, and $ x\notin P $, we obtain that $ \left\langle y\right\rangle  \subseteq P $, hence $ y\in P $. 
	
	Case2: Let $ y\notin h(L) $, that is $ y=y_{g_1}+y_{g_2}+\cdots+y_{g_k} $ for some homogeneous elements  $ y_{g_i} \in L_{g_i}, \ i=1,2,\ldots k$.
	
	 Now $ [x,y]=[x,y_{g_1}] +[x,y_{g_2}]+\cdots+[x,y_{g_n}]\in P$. Since $ x $ is a homogeneous element in $ L $, then each component in the combination of $ [x,y] $ belongs to a different homogeneous component. As $ P $ is graded ideal we get that $ [x,y_{g_i}]\in P $ where both $ x $ and $ y_{g_i}   $  are in $ h(L) $, for all $ i=1,\ldots k $. Using Case1, we have also that $ y\in P $. Hence $ (P:x)\subseteq P $. 
	
	$ \left(\Leftarrow \right)  $
	
	Suppose $ I,J $  is a graded ideals in $ L $ with $ [I,J]\subseteq P $. Assume that $ I \nsubseteq P $. Then there exist a homogeneous element $ x\in h(I)\subseteq h(L) $ with $ x \notin P $, (otherwise, as $ P  $ is graded, we find that $ I \subseteq P$, which is not the case). As  $ [I,J]\subseteq P $ we obtain that $ [x,J]\subseteq P $ which implies that $ J\subseteq (P:x)=P $.
\end{proof}

\begin{theorem}\label{image of graded prime}
	Suppose $ \varphi: L\longrightarrow L^{\prime} $ is a graded Lie epimorphism from the $ G $-graded Lie algebra $ L $ into the $ G $-graded algebra $ L^{\prime} $. Assume $ P $ denotes a graded prime ideal of $ L $ so that $ \mathrm{Ker}(\varphi)\subseteq P $. Then $ \varphi(P) $ is graded prime ideal of $ L^{\prime} $.
\end{theorem}

\begin{proof}
	Suppose $ I^{\prime},\ J^{\prime} $ considered be a graded ideal of $ L^{\prime} $ so that  $ [I^{\prime},\ J^{\prime}] \subseteq \varphi(P)$. As $ \varphi $ is onto, then there is a graded ideals  $ I=\varphi^{-1}(I^{\prime}), \text{ and }\  J=\varphi^{-1}(J^{\prime}) $.
	
	Note, $\mathrm{Ker}(\varphi)\subseteq I $ and $\mathrm{Ker}(\varphi)\subseteq J $, and since $\mathrm{Ker}(\varphi)\subseteq P $ we get that 
	
$ [I,J] \subseteq\varphi^{-1}(\varphi([I,J]))\subseteq \varphi^{-1}(\varphi(P))=P$. Thus $ I\subseteq P $ or $ J\subseteq P $, hence implying $ I^{\prime}\subseteq \varphi(P) $ or $ J^{\prime}\subseteq\varphi (P) $.

\end{proof}
A direct result of Theorem (\ref{image of graded prime}) is the following conclusion.

\begin{corollary}
	Assume $ \varphi: L\longrightarrow L^{\prime} $ to be a graded Lie epimorphism. Suppose $ P^{\prime} $ indicates a graded ideal of $ L^{\prime} $. If $ \varphi^{-1}(P^{\prime}) $ is graded prime Lie ideal of $ L $, then $ P^{\prime} $ is a graded prime Lie ideal of $ L^{\prime} $.
\end{corollary}

\begin{theorem}
	Suppose $ P $ is a graded prime ideal. Then $ P_e $ is graded prime ideal of $ L_e $.
\end{theorem}

\begin{proof}
	Suppose $ I_e,\ J_e $ are graded ideals in $ L_e $ so  that $[ I_e,\ J_e]\subseteq P_e $. Hence $[ I_e,\ J_e]\subseteq P $, which implying $ I_e\subseteq P $ or   $ J_e\subseteq P $. As $ P_e=P\cap L_e $, then $ I_e\subseteq P_e $ or   $ J_e\subseteq P_e $
\end{proof}

\begin{definition}
	An ideal $ Q $ is said to be semiprime if for any ideal $ H $ with  $ H^{\prime}=[H,H]\subseteq Q $ implies that $ H\subseteq Q $.
\end{definition}

\begin{definition}
	A graded ideal $ Q $ is considered to be graded semiprime ideal if for any graded ideal $ H $ with  $ H^{\prime}=[H,H]\subseteq Q $ implies that $ H\subseteq Q $.
\end{definition}

It is evident that any (graded) prime ideal is (graded) semiprime ideal.

\begin{definition}
 The (graded) ideal $ N $ of $ L $ is considered to be (graded) irreducible if $ N=H\cap K $ along with $ H,\ K $ (graded) ideals of $ L $ indicates that $ N=H \text{ or } N=K$.
\end{definition}
\begin{lemma}\label{prime is irreducble}
Any graded prime Lie ideal is graded irreducible ideal.
\end{lemma}

\begin{proof}
	Assume $ P $ be a graded prime ideal of $ L $ and  $ I,J $ considered to be graded ideals in $ L $ with $ P=I\cap J $. Then $ [I,J] \subseteq I\cap J=P$. As $ P $ is graded prime then $ I\subseteq P $ or $ J\subseteq P $.
\end{proof}
In the next fact, we follow \cite[Proposition 4]{Naoki}.

\begin{prop}
	An ideal $ P $ of $ L $ is considered to be graded prime ideal if and only if $ P $ is graded irreducible and graded semiprime	ideal.
\end{prop}

\begin{proof}
$ \Rightarrow $ 

 Evident by using Lemme \ref{prime is irreducble}.

$ \Leftarrow $

%Let $ P $ be graded semiprime and graded irreducible ideal.

Assume $ I,J $ be graded ideals in $ L $ so that $ [I,J]\subseteq P $. Define the ideal  $ H=(I+P)\cap (J+P) $. It is evident that $ P\subseteq H $. Since $ [I,J]\subseteq P $, we have that $ H^{\prime}\subseteq [I+P,J+P]\subseteq P $. As $ P $ is semiprime, then $ H\subseteq P $. Hence, $ H=(I+P)\cap (J+P)=P $. Therefore  $ P= (I+P)$ or $ P= (J+P)$, that is $ I\subseteq P $ or   $ J\subseteq P $.
\end{proof}

\begin{theorem}
	Assume $ Q $ considered to be a graded ideal. Then $ Q $ indicates graded semiprime ideal if and only if for any $ x
	\in h(L)  $ with $  \left\langle x\right\rangle^{\prime} \subseteq Q $ implies that $ x\in Q $ 
\end{theorem}

\begin{proof}
$ 	\Rightarrow $  Clear

$ \Leftarrow $ Suppose $ H $ considered to be a graded ideal in $ L $ with $ H^{\prime} \subseteq Q$. Assume that $ H\nsubseteq Q $, then there is a homogeneous element $ x\in H-Q $. Since $\left\langle x \right\rangle \subseteq H   $ , then $\left\langle x \right\rangle ^{\prime}\subseteq H^{\prime} \subseteq Q  $, which indicates that $ \left\langle x\right\rangle \subseteq Q $, that is $ x\in Q $, which is a contradiction.
\end{proof}

\section{Graded total prime Lie ideals}
\begin{definition}
	An ideal $ P $  of $ L $ is considered to be graded total prime ideal of $ L $ if for all $ x,y  \in h( L) $ with  $ \left[ x,y\right] \in P $  indicates that $ x\in  P $ or $ y\in  P $.
\end{definition}
Any graded prime ideal is clearly graded total prime.

%النظرية التالية غير مبرهنه عند ابو دواس . ساقوم بترجمتها و اثباتها ان شاء الله .

%\begin{theorem}	Let $ P $ be a graded ideal of $ L $. Then $ P $ is graded total prime ideal if and only if $  x,y\in h(L)$ with $\left\langle  [ x , y  ]\right\rangle \subseteq P $, then $ \left\langle x\right\rangle \subseteq P $ or $ \left\langle y\right\rangle \subseteq P $. 
		
		%\item$  (P:L)=(P:x) $ for all $ x\in h(L)-P $.
		%\item $ (P:L) $ is graded total prime ideal.	\end{theorem}

		\begin{theorem}
			Assume $ P $ considered to be a graded ideal of $ L $. Then  $ P $ is graded total prime ideal if and only if for any $  x,y\in h(L)$ with $\left\langle  [ x , y  ]\right\rangle \subseteq P $, then $ \left\langle x\right\rangle  \subseteq P $ or $  \left\langle y\right\rangle  \subseteq P $.
		\end{theorem}

\begin{proof}
$  \Rightarrow  $
	
	Let $ x,y\in h(L) $ with $ \left\langle [x,y]\right\rangle \subseteq P $. Hence $ [x,y] \in P $, which implies that $ x\in P  $ or $ y\in P $, that is $ \left\langle x\right\rangle \subseteq  P$ or $ \left\langle y\right\rangle \subseteq  P$. 
	
$ \Leftarrow $
	
	Let $ x,\ y \in h(L) $ with $[x,y] \in P $. Then $\left\langle [x,y] \right\rangle \subseteq P   $. Hence, $ \left\langle x\right\rangle \subseteq P $ or $ \left\langle x\right\rangle \subseteq P $, and so $ x\in P $ or $ y\in P $.
		%$ (2)\Rightarrow(3) $	Let $ x\in h(L)-P $. it is clear that  $ (P:L)\subseteq (P:x) $ because of the harder condition on $ (P:L) $. Now let $ y\in (P:x) $. Hence we can use the same argument in chosen $ y $ as in proof Theorem (\ref{P g.p iff (P:x)=p}), so without loss of generality, assume that $ y\in h(L) $. Since $ y\in (P:x) $ then $ [x,y]\in P $, which implies that $ \left\langle [x,y]\right\rangle \subseteq P $. Since $ x\notin P  $, then $ \left\langle y\right\rangle \subseteq P $. But $ [y,L]\subseteq \left\langle y\right\rangle  $, which means that $ y\in (P:L) $.

\end{proof}

\begin{theorem}
	Assume $ P $ indicates a graded total prime ideal. Then $ (P:x)=(P:L) $ for all $ x\in h(L)-P $.
\end{theorem}

\begin{proof}
Let $ x\in h(L)-P $. It is evident that  $ (P:L)\subseteq (P:x) $ due to the stricter condition placed on $ (P:L) $. Now let $ y\in (P:x) $. Hence we can use the same argument in chosen $ y $ as in proof Theorem (\ref{P g.p iff (P:x)=p}), so without loss of generality, assume that $ y\in h(L) $. Since $ y\in (P:x) $ implies $ [x,y]\in P $. Since $ x\notin P  $, then $ y\in  P $, that is  $ [y,L]\subseteq P  $, which means that $ y\in (P:L) $.
\end{proof}

\begin{theorem}
	Suppose $ P $ considered be a graded total prime ideal. Then $ (P:L) $  is graded total prime ideal of $ L $.
\end{theorem}
\begin{proof}
		
	Let $ x,\ y \in h(L) $ with $ [x,y]\in (P:L) $. Actually if $ L=P $ then we have done. So let $ P\neq L $ and choose $ z\in L-P $. Since  $ [x,y]\in (P:L) $, we have that $ [[x,y],z]\in P $. Hence $ [x,y]\in P $, that is $ x\in P $ or $ y\in P $. Which implies that $ [x,L]\subseteq P $ or $ [y,L]\subseteq P $.
\end{proof}
\begin{corollary}
		Suppose $ P $ considered to be a graded total prime ideal. Then $ (P:x) $  is graded total prime ideal of $ L $ for all $ x\in h(L)-P $.
\end{corollary}
\begin{definition}
	The nonempty subset $ S $ of $ h(L) $ is called graded-multiplicatively closed if when ever $ x,y\in S $ implies that $ [x,y] \in S $
\end{definition}

\begin{prop}
	A graded ideal $ P $ of $ L $ is graded  total prime ideal if and only if $ S=h(L)-h(P) $ is graded-multiplicatively closed.
\end{prop}

\begin{proof}

	$( \Rightarrow )$ Let  $ P $ indicates  a graded prime ideal, and  $ x_g, \ y_h \in S$, that are homogeneous elements in $ L $ which is not  belong to $ P $. it is evident that $ [x_g,y_h]\in h(L) $. If $ [x_g,y_h]\in h(P)\subseteq P $, then $ x_g\in P $ or $ y_h \in P $. As $ P $ is graded, then $ x_g\in P_g $ or $ y_h \in P_h $ which is a contradiction. That is $ [x_g,y_h]\in h(L)-h(P)=S $.

	$ (	\Leftarrow) $ Let $ S $ multipicatively-closed, and let $ x,y \in h(L) $ with $ [x,y]\in P $. If $ x $  and $ y $ are both in $ S $, and since $ x $ and $ y $ are homogeneous, then $ [x,y]\notin P $, a contradiction. Hence $ x\notin S $ or $ y\notin S $, that is $ x\in P $ or $ y\in P $.

\end{proof}

  \begin{theorem}\label{image of graded total prime}
  	Assume $ \varphi: L\longrightarrow L^{\prime} $ indicates a graded Lie epimorphism from the $ G $-graded Lie algebra $ L $ to the $ G $-graded algebra $ L^{\prime} $. Suppose $ P $ considered be a graded total prime ideal of $ L $ so that $ \mathrm{Ker}(\varphi)\subseteq P $. therefore, $ \varphi(P) $ is graded total prime ideal of $ L^{\prime} $.
  \end{theorem}
  
  \begin{proof}
  	Assume $x, y\in h(L^{\prime}) $ so that  $ [x,y] \subseteq \varphi(P)$. As $ \varphi  $ is onto then  there exists a homogeneous elements  $ a,b \in L $ so that $ \varphi(a)=x $ and $ \varphi(b)=y $. Moreover, there exists $ c\in P $ so that $ \varphi(c)=[x,y] $. Hence $\varphi([a,b])=[\varphi(a),\varphi(b)=[x,y]=\varphi(c)]$, which implies that $ [a,b]-c\in \mathrm{Ker}(\varphi)\subseteq P $. Since $ c\in P $, then $ [a,b]\in P  $, obtains that $ a\in P $ or $ b\in P $. Thus, $ x\in \varphi(P) $  or $ y\in \varphi(P) $.

  \end{proof}
  
 Furthermore, Theorem (\ref{image of graded total prime}) directly leads to the following conclusion..
  \begin{corollary}
  	Suppose $ \varphi: L\longrightarrow L^{\prime} $ considered to be a graded Lie epimorphism. Assume $ P^{\prime} $ indicates a graded ideal of $ L^{\prime} $. If $ \varphi^{-1}(P^{\prime}) $ is graded total prime ideal of $ L $, then $ P^{\prime} $ is a graded total prime ideal of $ L^{\prime} $.
  \end{corollary}


\begin{thebibliography}{ABFP08}
	
	\bibitem{g.p submod non comm }Abu-Dawwas, Rashid, Malik Bataineh, and Maaly Al-Muanger. “Graded Prime Submodules over Non-Commutative Rings.” Vietnam Journal of Mathematics 46, no. 3 (2018): 681–92. doi:10.1007/s10013-018-0274-2. 
	
	
	
	
	\bibitem{semi Dawwas}Abu-Dawwas, Rashid. “Graded Semiprime and Graded Weakly Semiprime Ideals.” Italian Journal of Pure and Applied Mathematics 36 (2016): 535–42. 
	
	
	
	
	\bibitem{batainah 1}Anderson, D. D., and Malik Bataineh. “Generalizations of Prime Ideals.” Communications in Algebra 36, no. 2 (2008): 686–96. doi:10.1080/00927870701724177. 
	
	%\bibitem[2]{a.p in noncomm}A. Abouhalaka and S. Findik, Almost prime ideals in non-commutative rings, (2022),	arXiv:2201.09758v1. 
	
	
	\bibitem{g.p submod  }Atani, Shahabaddin. “On Graded Prime Submodules.” Chiang Mai J. Sci 33, no. 1 (2006): 3–7. 
	
\bibitem{BS2}Bahturin, Yuri, and Abdallah Shihadeh. “Graded Torsion-Free $ \s $-Modules of Rank 2.” Journal of Algebra and Its Applications 21, no. 11 (2021). doi:10.1142/s0219498822502292. 



\bibitem{BS3}Bahturin, Yuri, and Abdallah Shihadeh. “Tensor Products of Graded-Simple $ \s $-Modules.” Turkish Journal of Mathematics 46, no. 5 (2022): 1619–43. doi:10.55730/1300-0098.3222. 

\bibitem{BKS}Bahturin, Yuri, Mikhail Kochetov, and Abdallah Shihadeh. “Graded Modules over Simple Lie Algebras.” Revista Colombiana De Matemáticas 53, no. supl (2019): 45–86. doi:10.15446/recolma.v53nsupl.84006. 
	








\bibitem{semi farz}Farzalipour, F, and P Ghiasvand. “On Graded Semiprime and Graded Weakly Semiprime Ideals.” International Electronic Journal of Algebra 13 (2013): 15–22. 


\bibitem{Naoki} Kawamoto, Naoki. “On Prime Ideals of Lie Algebras.” Hiroshima Mathematical Journal 4, no. 3 (1974). doi:10.32917/hmj/1206136843. 









\bibitem{g radicals }Refai, M, M Hailat, and S Obiedat. “ Graded Radicals and Graded Prime Spectra.” Far East J. Math. Sci 1 (2000): 59–73. 




%\bibitem[JBK08]{almost graded}A. Jaber, M. Bataineh and H. Khashan, Almost graded prime ideals, Journal of Mathematicsand Statistics, 4 (4) (2008), 232-236.
	
%\bibitem[4]{birth of a.p}M. S. Bhatwadekar, P. K. Sharma, Unique factorization and brith of almost prime, Communications in Algebra, 33 (1) (2005), 43-49.


%\bibitem[12]{a.p g ideals}Jaber, A.; Bataineh, M.; Khashan, H. Almost graded prime ideals. J. Math. Stat. 2008, 4, 231–235.

%\bibitem[14]{g.a. p  non comm   }J. Shtayat,  R. Abu-Dawwas, G. Bani Issa,  Graded Almost Prime Ideals over Non-Commutative Graded Rings
%\bibitem[AP73]{AP73} Arnal, D. Pinczon, G. {\it Ideaux à gauche dans les quotients simples de l'algèbre enveloppante de sl(2)}. Bulletin de la société mathématique de France, 1973, 101, pp.381. 

%\bibitem[AP74]{on algeb irr}  Arnal, D. Pinczon, G. {\it On algebraically irreducible representations of the Lie algebra sl(2)}, J. Math. Phys. {\bf 15} (1974), 350–359.

%\bibitem[BK10]{typeABCD}Bahturin, Y. and Kochetov, M. {\it Classification of group gradings on simple Lie algebras of}
%types $ A $, $ B $, $ C $ and $ D $. J. Algebra {\bf 324} (2010), no. 11, 2971–2989.

%	\bibitem[BM04]{BM}Bahturin, Y.; Molev, A., { \it Casimir elements for some graded Lie algebras and superalgebras}, Czech. J. Phys., {\bf 54} (2004), 1159 - 1164. 
%\bibitem[BSZ01]{Yassociative}Bahturin, Y.; Sehgal, S. and Zaicev, M.{ \it Group gradings on associative algebras}, J. Algebra, 241(2001), 667-698.

	%\bibitem[Bav90]{Bavula classif} Bavula, V., {\it Classification of simple sl(2)-modules and the finite dimensionality	of the module of extensions of simple sl(2)-modules}, Ukrainian Math. J. {\bf 42} (1990), no. 9, 1044--1049.
	
	

%	\bibitem[Bav92]{Bav weyl}  Bavula, V., { \it Generalized Weyl algebras and their representations}, Algebra and Analiz, {\bf 4(1) }(1992), 75--97; English transl. in St. Petersburg Math. J., {\bf 4(1)} (1993), 71--92.
%\bibitem[BL07]{BL} Billig, Y.; Lau, M. \emph{Thin coverings of modules}, J.~Algebra \textbf{316} (2007), no.~1, 147--173.
%	\bibitem[Blo81]{irr..rep..block}	Block, R., {\it The irreducible representations of the Lie algebra $\mathfrak{sl}_2$ and of the	Weyl algebra}, Adv. Math., {\bf 39} (1981), 69--110. 
	
	
	
%\bibitem[Cat98]{ideal sl2 } Catoiu, S. 1998.{ \it Ideals of the Enveloping Algebra U($ \mathfrak{sl}_{2} $)}. J. of Algebra, {\bf 202}: 142–177. 



%	\bibitem[EK13]{gradbook} Elduque, A.; Kochetov, M., Gradings on simple Lie algebras, Math. Surv. Mon., {\bf 189}, AMS, Providence, RI, 2013.
	
	
	
	
%	\bibitem[EK15]{EK_Israel} Elduque, A.; Kochetov, M., \emph{Graded modules over classical simple Lie algebras with a grading}, Israel J. Math. \textbf{207} (2015), 229--280.





%\bibitem[EK17]{EK_loop} Elduque, A.; Kochetov, M., {\it Graded simple modules and loop modules}, in: Groups, rings, group rings, and Hopf algebras,Contemp. Math., \textbf{688} (2017), 53--85.






%\bibitem[Hum78]{Hum} Humphreys, J. {\it Introduction to Lie Algebras and Representation Theory}, Graduate Texts in Mathematics, vol. 9, Springer-Verlag, New York, 1978, second printing, revised. MR MR499562 (81b:17007)
%\bibitem[Koc09]{Ksimple}
%Kochetov, M.	{\it Gradings on finite-dimensional simple Lie algebras}
%Acta Appl. Math. {\bf 108} (1) (2009), pp. 101-127




	
%\bibitem[MB16]{arb rank} Martin, F.; Prieto, C., {\it Construction of simple non-weight $ \mathfrak{sl}(2) $-modules of arbitrary rank}, J. Algebra, {\bf 472} (2017), 172--194.




%\bibitem[Maz09]{Mazorchuk}  Mazorchuk, V.,  Lectures on $ \s $-modules, World Scientific, 2009. 

%	\bibitem[Nil15]{u(h)struct} Nilsson, J. {\it Simple $ \mathfrak{sl}_{n+1} $-module structures on $ \mathcal{U}(\mathfrak{h}) $}, J. Algebra, {\bf 424} (2015), 294--329.

\end{thebibliography}
\end{document}